\documentclass[10pt]{article}
\usepackage{amsmath}
\usepackage{amssymb}

\usepackage{graphicx}

\usepackage{cite}

\usepackage{color} 

\usepackage{setspace} 
\usepackage[labelfont=bf,labelsep=period,justification=raggedright]{caption}

\date{}

\usepackage{latexsym,array,delarray,amsthm,amssymb,epsfig,setspace}
\usepackage{subfigure}
\theoremstyle{plain}
\newtheorem{thm}{Theorem}

\newtheorem{prop}[thm]{Proposition}
\newtheorem{cor}[thm]{Corollary}

\newtheorem {example} {Example}
\theoremstyle{definition}
\newtheorem{defn}[thm]{Definition}
\newtheorem{ex}[thm]{Example}

\theoremstyle{remark}
\newtheorem*{rmk}{Remark}

\newcommand{\rr}{\mathbb{R}}

\newcommand{\bfa}{\mathbf{a}}
\newcommand{\bfb}{\mathbf{b}}
\newcommand{\bfc}{\mathbf{c}}

\newcommand{\cF}{\mathcal{F}}

\newcommand{\ind}{\mbox{$\perp \kern-5.5pt \perp$}}

\newcommand{\comment}[1]{}

\begin{document}

% Title must be 150 characters or less
\begin{flushleft}
{\Large
\textbf{Structural identifiability of viscoelastic mechanical systems}
}
% Insert Author names, affiliations and corresponding author email.
\\
Adam Mahdi$^1$, 
Nicolette Meshkat$^1$, 
Seth Sullivant$^1$
\\
Department of Mathematics, North Carolina State University, NC, USA
\\
$^1$ These authors contributed equally to this work.\\
$\ast$ E-mail: Corresponding amahdi@ncsu.edu
\end{flushleft}

% Please keep the abstract between 250 and 300 words
\section*{Abstract}
We solve the local and global structural identifiability problems for  viscoelastic mechanical models represented by  networks of springs and dashpots.  We propose a very simple characterization of both local and global structural identifiability based on \emph{identifiability tables}, with the purpose of providing a guideline for constructing arbitrarily complex, identifiable spring-dashpot networks. We illustrate how to use our results 
in a number of examples and point to some applications in cardiovascular modeling.

% Please keep the Author Summary between 150 and 200 words
% Use first person. PLoS ONE authors please skip this step. 
% Author Summary not valid for PLoS ONE submissions.   
%\section*{Author Summary}

\section*{Introduction}
%\noindent {\bf Modeling and Identifiability.} 
Mathematical modeling is a prominent tool used to better understand complex mechanical or  biological systems \cite{Ottesen2011-Micro}.  
A common problem that arises when developing a model of a biological or mechanical system is that some of its parameters are unknown. 
This is especially important when those parameters have special meaning but cannot be directly measured. Thus a natural question arises: Can all, or at least some, of the model's parameters be estimated indirectly and \emph{uniquely} from observations of the system's input and output? This is the question of \emph{structural identifiability}. Sometimes the uniqueness holds only within a certain range.  In this case, we say that a system is only \emph{locally} structurally identifiable. 
There are numerous reasons why one would be interested in establishing identifiability. Structural identifiability is a necessary condition for the practical or numerical identifiability problem, which involves parameter estimation with real, and often noisy, data.  The unobservable biologically meaningful parameters of a model can only be determined (or approximated) if the model is structurally identifiable.  Moreover, optimization schemes cannot be employed reliably since they will find difficulties when trying to estimate unidentifiable parameters \cite{Banga2011-comp}.  The concept of structural identifiability was introduced for the first time in the work of Bellman and {\AA}str\"om \cite{BelAst70}. Since then, numerous techniques have been developed to 
analyze the identifiability of
 linear and nonlinear systems with and without controls \cite{WalLec81, VidBla00, VidBlaNoi01, Banga2011-comp, LitHeiLi10}; see also \cite{Miao11} for a review of different approaches.   

\smallskip

%\noindent {\bf  viscoelastic models.} 
Viscoelastic mechanical models that utilize springs and dashpots in various configurations have been widely used in numerous areas of research including material sciences \cite{AnaAme06},  computer graphics \cite{TerFle1988}, and  biomedical engineering to describe mechanical properties of biological systems  \cite{Bland1965, Chris71-VEintro, Ros50, AkyJonWal90, GanCho96, FungMec, BurPLoS, Ack2013}.  To achieve a desirable response, networks with different numbers of springs and dashpots in various configurations have been constructed.  For example, it is well-known that the simplest models of viscoelastic materials such as Voigt (spring and dashpot in parallel) or Maxwell (spring and dashpot in series) do not offer satisfactory representation of the nature of real materials \cite{DieLekTur1998}. Thus more complicated configurations are usually constructed and analyzed \cite{BurPLoS}. 

\smallskip

%\noindent {\bf Identifiability of viscoelastic models.} 
In this paper we investigate the identifiability problem of viscoelastic models represented by an arbitrarily complex spring-dashpot network.  Although there exist numerous methods that can determine the type of identifiability of a system of ordinary differential equations, generally they are difficult to apply. Our results will show in a remarkably simple way how to verify whether the studied model is (locally or globally) structurally identifiable. In case it is unidentifiable, our method provides an explanation why  this is the case and how to reformulate the problem.  Moreover, the existing methods usually allow to establish the identifiability only \emph{a posteriori}, i.e.~after concrete systems have been established. Thus, we also introduce ``identifiability tables'', which allow not only to check but also to construct an arbitrarily complex \emph{identifiable} spring-dashpot network.

\begin{figure}[t!]
\begin{center}
\includegraphics[clip,trim={0 12cm 0 0},width=8.7cm]{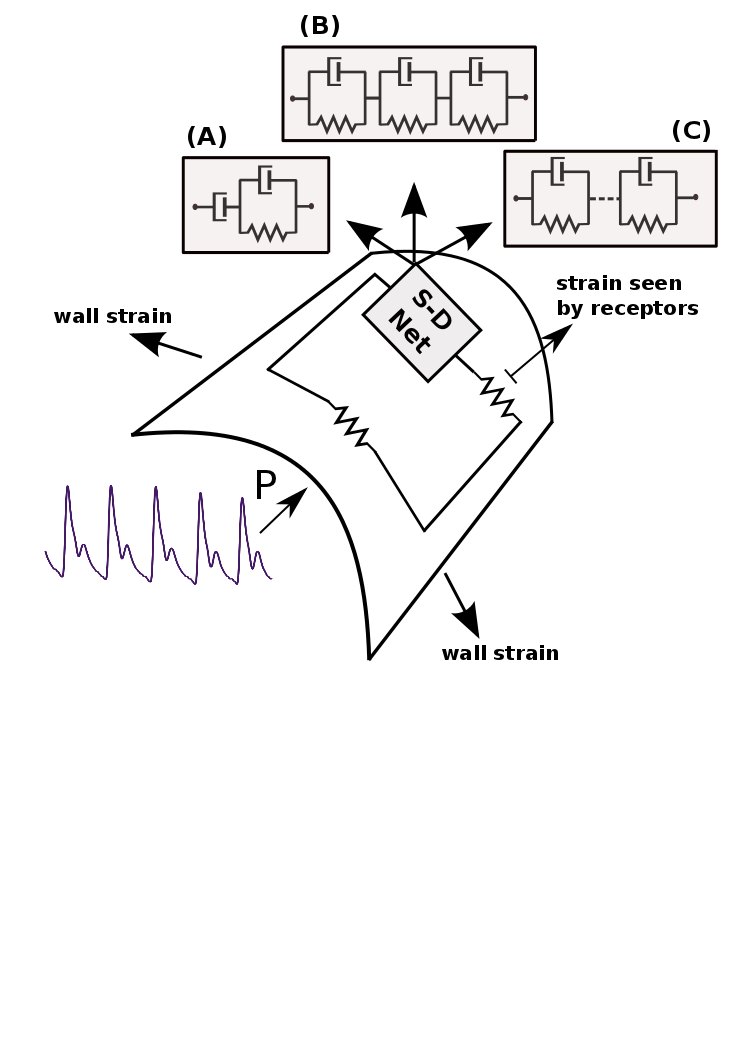}
\end{center}
\caption{Changing blood pressure (P) causes periodic expansion and contraction of the arterial wall. Spring-dashpot (S-D) networks  are often used in order to describe the biomechanical properties of the arterial tissue as well as the strain sensed by various receptors (e.g.~baroreceptors) embedded in the arterial wall. Typically a spring (representing a receptor's nerve ending) is combined in series with a S-D network (representing viscoelastic coupling of the nerves to the wall). Recently, several cardiovascular approaches have used the framework described above, in particular, choosing one of the following S-D networks: (A) Burgers-type model \cite{AlfPhD}; (B) three element Kelvin-Voigt body \cite{BugCowBea10}; (C) generalized Kelvin-Voigt model \cite{MahStuOttOlu13}.}\label{Fig:CV}
\end{figure}

\subsubsection*{Application to cardiovascular modeling}
A particular motivation for this work comes from cardiovascular modeling \cite{BugCowBea10, MahStuOttOlu13},  although the results of this paper can be applied to any viscoelastic modeling approach. 
\smallskip
 
\noindent{\it Arterial wall.} 
Changing blood pressure causes periodic expansion and contraction of the arterial wall (see Fig.~\ref{Fig:CV}). It is well-known that the stress-strain curves of the artery walls exhibit hysteresis, which is understood to be a  consequence of the fact that the wall is viscoelastic. Another manifestation of the viscoelasticity of the arterial tissue is the stress relaxation experiments under constant stretch (strain). Spring-dashpot (S-D) networks  are often used in order to describe the biomechanical properties of the arterial tissue  \cite{LeaTay1966, McDonaldsFlow, KalSch08}.  Identifiable networks can be determined using the results of this paper (see Theorem \ref{thm:parameqcoeff}).  

\smallskip

\noindent{\it Neural activity.} It is common to use the spring-dashpot network to describe the neural firing of various sensors (e.g.~muscle spindle, baroreceptors), see \cite{Houk66, Hasan83,AlfPhD, BugCowBea10}. Typically one assumes that the firing activity is proportional to the strain sensed by a spring connected in series with a spring-dashpot network, which represents a local integration of the nerve endings to the arterial wall (see Fig.~\ref{Fig:CV}).  Then the arterial wall and neural activity models are combined. Although separately each model is structurally identifiable, there is no guarantee that the resulting viscoelastic structure is identifiable. Thus, using our results given in  Theorem \ref{thm:localMT}, we can establish whether the combined viscoelastic model is identifiable, and if not, what needs to be modified.

%%%%%%%%%%%%%%%%%%%%%%%%%%%%%%%%%%%%%%
% Results and Discussion can be combined.
\section*{Results and Discussion}

After reviewing basic concepts of viscoelasticity of systems, we present and discuss our main results related to local and global structural identifiability of such systems. Finally, we illustrate our results with a number of  examples from the literature.

\medskip
%---------------------------------------------------------
%\subsection*{Definition and motivating example}

\subsubsection*{Spring-dashpot networks}

The ideal linear elastic material follows  Hooke's law $\sigma=E\epsilon$, where $E$ is a Young's modulus (or a spring constant), which describes the relationship between the stress $\sigma$ and the strain $\epsilon$. Analogously, the relation $\sigma= \eta \dot \epsilon$ describes the viscous material,  where $\dot\epsilon = d\epsilon/dt$ and $\eta$ is a viscous constant \cite{Flu75}. In the basic linear viscoelasticity theory, the elastic and viscous elements are combined. In this work, we shall be concerned with the problem of identifiability of networks of springs and dashpots that are essentially one-dimensional. The elements can be combined either in series or in parallel. In order to obtain the relationship between the total stress (force) $\sigma$ and the total strain (extension) $\epsilon$ for a given spring-dashpot network, we use two fundamental rules. For two viscoelastic elements connected in series, the stress is the same in both elements, but the total strain is the sum of individual strains on each element. On the other hand, for elements connected in parallel, the strain is the same for both elements, but the total stress is the sum of individual stresses on each element. Now we consider concrete viscoelastic networks, starting with the simplest configurations.

\begin{figure}
\begin{center}
\includegraphics[width=10cm]{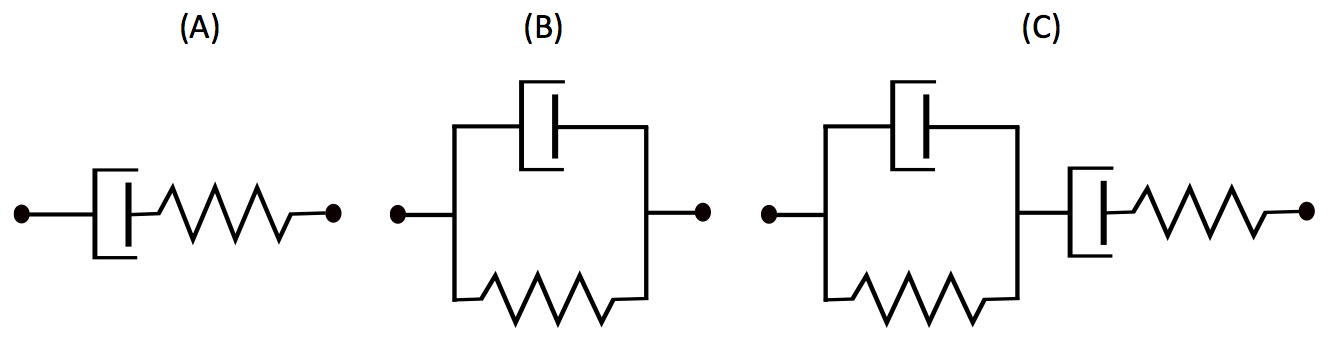}
\end{center}
\caption{Simple linear viscoelastic models. (A) Maxwell element, (B) Voigt element, (C) Burgers model.}\label{Fig:simple}
\end{figure}

\begin{example}[Maxwell element]\upshape\label{ex:M}
The series combination of a spring, denoted by its constant $E$, and a dashpot, denoted by its constant $\eta$,  is known as a Maxwell element (see Fig.~\ref{Fig:simple}(A)). Since the elements are connected in series, the stress $\sigma$ is the same on both elements and the total strain $\epsilon$ is the sum of strains $\epsilon_E$ and $\epsilon_\eta$ corresponding to the spring and dashpot, respectively. Now,  the relationship between the total strain and stress for this system is 
\begin{equation}\label{ce:M}
\dot\epsilon = \dot \sigma/E+\sigma/\eta.
\end{equation}
\end{example}

\begin{example}[Voigt element]\upshape\label{ex:V}
Another simple example is the Voigt element (also known as Kelvin or Kelvin-Voigt)  given in Fig.~\ref{Fig:simple}(B). Following the steps outlined in the previous example, we obtain the $\epsilon-\sigma$ relationship
\begin{equation}\label{ce:V}
E\epsilon+\eta \dot \epsilon=\sigma.
\end{equation}
\end{example}

\begin{example}[Burgers model]\upshape\label{ex:Burgers}
In our third example we consider a particularly popular four-element model, represented by a Maxwell element combined in series with a Voigt element, and known as the Burgers model ( Fig.~\ref{Fig:simple}(C)). Denote by subscript $m$ and $v$ the spring and viscous constants of the Maxwell and Voigt elements, respectively. Note that the stress $\sigma$ is the same on all three elements connected in series (Voigt, spring and dashpot). Eliminating the corresponding local strains, we obtain the following relationship
\begin{equation}\label{ce:Burgers}
E_m\ddot\epsilon + \frac{E_m E_v}{\eta_v}\dot\epsilon=\ddot \sigma + \Big[ \frac{E_m}{\eta_m} +\frac{E_m}{\eta_v}+\frac{E_v}{\eta_v}\Big]\dot\sigma+\frac{E_m E_v}{\eta_m\eta_v}\sigma.
\end{equation}
\end{example}

\medskip
%---------------------------------------------------------
%\subsection*{Characterization theorem}

\subsubsection*{Identifiability characterization}

First note (cf.~Examples \ref{ex:M}, \ref{ex:V}, and \ref{ex:Burgers}) that for any configuration of springs $\bar{E}=(E_1,...,E_N)$ and dashpots $\bar \eta=(\eta_1,...,\eta_M)$, the total strain--stress relationship can always be written as the following $(n+1)$-st order linear ordinary differential equation 
\begin{equation}\label{eq:ODE}
a_{n+1} \epsilon^{(n+1)}+a_n \epsilon^{(n)}+\cdots+a_0 \epsilon = b_n \sigma^{(n)}+\cdots+b_0 \sigma,
\end{equation}
where the coefficients $a_j=a_j(\bar E,\bar \eta)$ and $b_k=b_k(\bar E,\bar \eta)$ are functions of the spring and dashpot constants.  The precise value of $n$ and the
forms of  $a_j(\bar E,\bar \eta)$ and $b_k(\bar E,\bar \eta)$ will depend
on the particular structure of the spring-dashpot model.
Equation \ref{eq:ODE} is known as the \textit{constitutive equation}. In the context of spring-dashpot networks, identifiability concerns whether or not it is possible to recover the unknown parameters ($\bar E$ and $\bar \eta$) of the system from the governing equation of the model, given only the total stress $\sigma$ and total strain $\epsilon$. In other words, we assume that we know the stress and the strain at the bounding nodes only and ask if it is possible to determine the unknown parameters ($\bar E$ and $\bar \eta$).
In order to uniquely fix the coefficients of the constitutive equation \eqref{eq:ODE}, we require that \eqref{eq:ODE} be \emph{normalized} so that the leading term (in $\sigma$ or $\epsilon$, depending on the situation) is monic.  Thus, letting the $d$ non-monic coefficients of \eqref{eq:ODE} be represented by the vector $\bfc  =  (\bfa(\bar E, \bar \eta), \bfb(\bar E, \bar \eta))$, we have the following formal definition of identifiability.

\begin{defn} \label{defn:id} 
Let $\bfc$ be a function $\bfc:\Theta\rightarrow{\rr^{d}}$, where
$\Theta \subseteq \rr^{N+M}$ is the parameter space.
The model is \emph{globally identifiable} from $\bfc$ if and only if the map $\bfc$ is one-to-one.  The model is
\emph{locally identifiable} from $\bfc$ if and only if the map $\bfc$ is finite-to-one.  
The model is
\emph{unidentifiable} from $\bfc$ if and only if the map $\bfc$ is 
infinite-to-one.
\end{defn}
Note that local identifiability is equivalent to saying that around
each point in parameter space there exists a neighborhood on which the 
function $\bfc$ is one-to-one. For example, for the Burgers model considered in Example \ref{ex:Burgers}, the coefficient function $\bfc:\rr^4\to\rr^4$ is defined as
\[
{\bf c}:(E_m,E_v,\eta_m,\eta_v)\to\Big(E_m, \frac{E_m E_v}{\eta_v},  \frac{E_m}{\eta_m} +\frac{E_m}{\eta_v}+\frac{E_v}{\eta_v}, \frac{E_m E_v}{\eta_m\eta_v}\Big).
\]
Technically speaking, in this paper we will consider the slightly weaker notion
of \emph{generic global identifiability} (or generic local identifiability, or generic unidentifiability), where \textit{generic} means that the property holds almost everywhere.
We will omit the use of the term generic when speaking of identifiability. 

Definition \ref{defn:id} implies that if there are more parameters than non-monic coefficients, then the system must be unidentifiable.  Thus, a necessary condition for structural identifiability is that the number of parameters $\bar E, \bar \eta$ (elements of the network) is less than or equal to the number of non-monic coefficients in the constitutive equation \eqref{eq:ODE}. We will soon show that the number of non-monic coefficients is bounded by the number of parameters in spring-dashpot networks.  Thus, in this case, a necessary condition for structural identifiability is that the number of parameters and non-monic coefficients are equal.  We will prove that, remarkably, in the case of viscoelastic models represented by a spring-dashpot network, the converse to this statement holds as well.

\begin{thm} [Local identifiability] \label{thm:parameqcoeff} 
A  viscoelastic model represented by a spring-dashpot network is locally identifiable if and only if the number of non-monic coefficients of  the corresponding constitutive equation \eqref{eq:ODE} equals the total number of its moduli $E_j$ and viscosity parameters $\eta_k$.
\end{thm}

Note that although the constitutive equation \eqref{eq:ODE} is a linear differential equation,
its coefficients considered as functions of spring and viscous constants are not linear
functions of the parameters (see \eqref{ce:Burgers}). 
 Thus, Theorem \ref{thm:parameqcoeff} allows to reduce the difficult problem of checking
 one-to-one or finite-to-one behavior of nonlinear functions  to simply counting the number of parameters (springs and dashpots) and 
non-monic coefficients of the constitutive equation and asking whether the two numbers are equal. The positive answer implies local identifiability, whereas a negative answer implies unidentifiability.  Consider, for example, the Maxwell and Voigt elements, and the Burgers model. We note that the constitutive equations \eqref{ce:M}, \eqref{ce:V}, and \eqref{ce:Burgers} for all three models are already in the normalized form. Now, simply by counting the number of parameters and the non-monic coefficients of the constitutive equations, we  see that the two are equal for each model. Thus, by the above theorem, all three models are locally structurally identifiable. 

\begin{figure}[t!]
\begin{center}
\includegraphics[width=8cm]{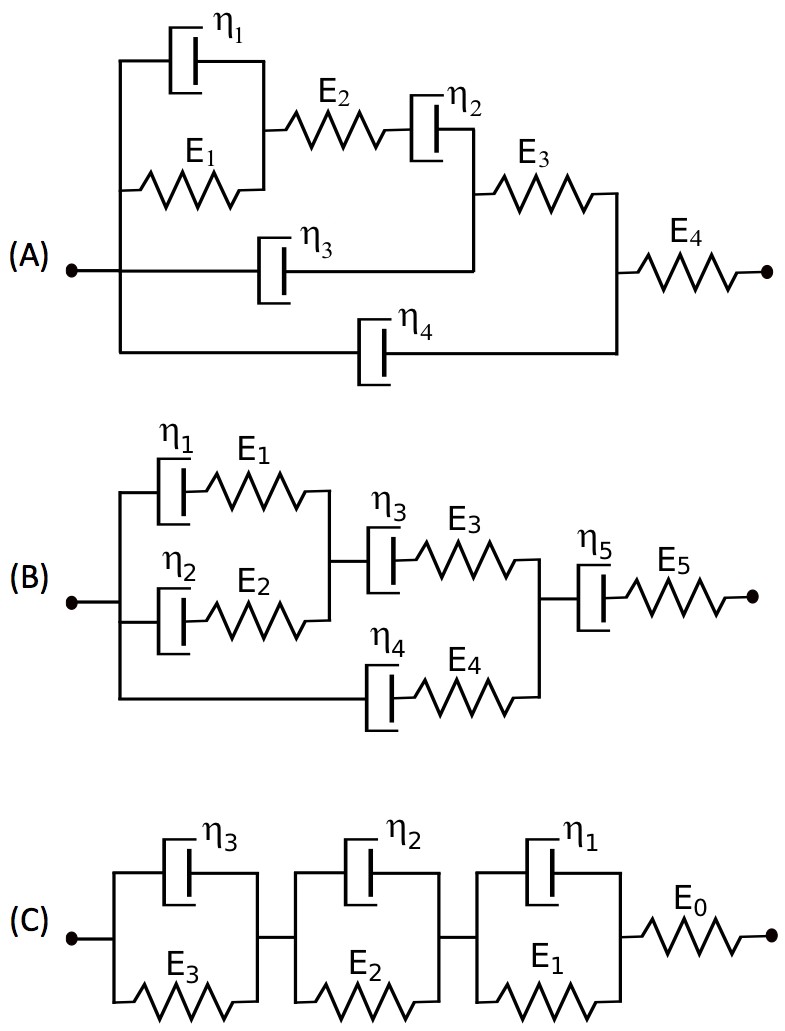}
\end{center}
\caption{(A) Multi-parameter linear viscoelastic model considered by Dietrich et al. \cite{DieLekTur1998}. (B) Ten element viscoelastic model studied in \cite{Ros50}, (C) A viscoelastic model of used to describe the baroreceptor nerve ending coupling to the arterial wall (see  \cite{BugCowBea10} and \cite{MahOttOlu12, MahStuOttOlu13}).}\label{Fig:turski}
\end{figure}

\medskip
%---------------------------------------------------------
%\subsection*{Constructive theorem}

\subsubsection* {Constructing identifiable models} 
Now we examine when combining two identifiable models results also in an identifiable model. This will allow us to construct arbitrarily complex and identifiable spring-dashpot networks.  

\smallskip

We start with an observation, which we prove in the following section, related to the possible form of any differential equation that describes a spring-dashpot network. 
\begin{prop}\label{Prop:types}
Every spring-dashpot network, given by equation \eqref{eq:ODE}, has one of the four possible types
\begin{equation}\label{types}
\begin{aligned}
&{\bf Type\,\, A:}\quad b_{0}b_{n} \ne 0, \quad a_{n+1}=0,\quad  a_0a_n\ne 0\\
&{\bf Type\,\,  B:}\quad b_{0}b_{n} \ne 0, \quad   a_{0}=0,\quad a_1a_{n+1}\ne 0\\
&{\bf Type\,\,  C:}\quad b_{0}b_{n} \ne 0, \quad a_0a_{n+1}\ne 0\\
&{\bf Type\,\,  D:}\quad b_{0}b_{n} \ne 0, \quad a_{n+1}=a_0=0,\quad a_1a_n\ne 0.
\end{aligned}
\end{equation}
\end{prop}

Recall that $n$ is the highest derivative of the stress component $\sigma^{(n)}$, which appears in the total strain-stress equation \eqref{eq:ODE}.  Now we illustrate the different types of networks defined in the above proposition by considering the simplest elements.  

\begin{ex}
For  a spring, given by  $E\epsilon=\sigma$,  we have $n=0$ (only $\sigma$ appears in the constitutive equation, but none of its derivatives),  $a_1=a_{n+1}=0$, and $a_0=a_n=E\ne0$. Therefore a spring is of type A. Note that for a dashpot, which is given by $\eta \dot \epsilon = \sigma$, we also have  $n=0$ but $a_0=0$ and $a_1=a_{n+1}=\eta\ne 0$. Thus,  according to notation given in Proposition \ref{Prop:types}, a dashpot is of type B. 
For the Voigt element, given by \eqref{ce:V}, we have $n=0$ as well as $a_0=E$ and $a_1=a_{n+1}=\eta$ (that is  $a_0a_{n+1}\ne 0$). We conclude that it is of type C. Finally, a Maxwell element is given by \eqref{ce:M}. Note that here  $n=1$ (since $\dot \sigma$ appears in \eqref{ce:M}),  $a_{0}=0$, $a_{2}=a_{n+1}=0$, and $a_1=a_n=1\ne 0$. Thus a Maxwell element is of type D.
\end{ex}

Once the constitutive equation has been determined for a given spring-dashpot network, it is very easy to establish the type that it belongs to. Unfortunately, $n$ does not always have a physical significance. The value of $n$ is determined by the specific network and cannot be easily related to the number of springs and dashpots as we will illustrate later on.

\begin{thm}[Local identifiability]\label{thm:localMT}
Consider two locally identifiable spring-dashpot systems $N_{1}$ and $N_{2}$ of one of the four types $A$, $B$, $C$, $D$.  Then the new model resulting in joining $N_{1}$ and $N_{2}$ either in parallel or in series is of the type indicated by the Identifiability Tables (Table \ref{thetable}). 
The letter \emph{u} indicates that the network is unidentifiable, otherwise it is identifiable of the given type.
\end{thm}

There are several ways one could use the above theorem. One way is to establish the local identifiability of a given spring-dashpot network. Contrary to our similar result given in Theorem \ref{thm:parameqcoeff}, this can be done without actually calculating
the constitutive equation.  We will show how to apply Theorem \ref{thm:localMT} to establish structural identifiability after first introducing some notation. 
Given any  two spring-dashpot models $M$ and $N$, we use the following notation 
$(M \lor N)$ and   $(M \land N)$ to denote respectively the parallel  and series combination of $M$ and $N$.  Let $\cF$ denote the function that takes a spring and dashpot model
$M$ and outputs its type ($A,B,C,D$) if it is locally identifiable, and $u$ if
it is unidentifiable.  To apply $\cF$ to a complicated model built up from 
springs and dashpots using series and parallel connections, we replace any springs and dashpots with their respective types $A$ and $B$  as well as the operations $\lor$ and $\land$ with $\oplus$ and $\odot$, respectively. Then we
apply the operations in the Identifiability Tables (see Table \ref{thetable}).

\begin{table}
\caption{
\bf{Identifiability Tables.}}
\hspace{2cm}
\subtable[Parallel connection]{
\begin{tabular}{c|ccccc}
$ \oplus$&{\bf A}&{\bf B}&{\bf C}&{\bf D}& {\bf u}\\ \hline
{\bf A}&u&C&u&A&u\\
{\bf B}&C&u&u&B&u\\
{\bf C}&u&u&u&C&u\\
{\bf D}&A&B&C&D&u\\
{\bf u}&u&u&u&u&u
\end{tabular}
}
\hspace{2cm}
\subtable[Series connection]{
\begin{tabular}{c|ccccc}
$\odot$&{\bf A}&{\bf B}&{\bf C}&{\bf D}& {\bf u}\\ \hline
{\bf A}&u&D&A&u&u\\
{\bf B}&D&u&B&u&u\\
{\bf C}&A&B&C&D&u\\
{\bf D}&u&u&D&u&u\\
{\bf u}&u&u&u&u&u
\end{tabular}
}
\begin{flushleft}
When connecting two identifiable spring-dashpot networks of one of the types $A$, $B$, $C$, $D$, or an unidentifiable $u$  either in series or in parallel, the above tables establish the type of the resulting identifiable system. If the resulting structure is unidentifiable it is indicated by $u$. For example, a parallel connection of two networks of types $A$ and $D$ gives rise to an identifiable network of type $A$ (see (a)), but the series connection results in an unidentifiable structure (see (b)).
\end{flushleft}\label{thetable}
\end{table}

\begin{example}[Local identifiability of the Maxwell element]\upshape
Note that the Maxwell model shown in Fig. \ref{Fig:simple}(A) can be symbolically written as 
\[
M  = E\land \eta.
\]  
In this formula,  we simply replace the spring and  the dashpot with $A$ and $B$, respectively, as well as the operations $\lor$ and $\land$ with $\oplus$ and $\odot$, respectively, to obtain 
\[
\cF(M)=(A\odot B)= D. 
\]
Thus we conclude that the Maxwell model is locally identifiable and is of type $D$.
\end{example}

\begin{example}[Local identifiability of the Burgers model]\upshape\label{Ex:burgers}
Similarly, the Burgers model shown in Fig. \ref{Fig:simple}(C) can be symbolically written as 
\[
M = (E_v\lor \eta_v)\land(E_m\land \eta_m).
\]
To check the local identifiability, we find $\cF(M)$ and use Table \ref{thetable} to obtain
\[
\cF(M)=\overbrace{(A\oplus B)}^C\odot\overbrace{(A\odot B)}^D
= C\odot D=D.
\]
We conclude that the Burgers model is locally identifiable and of type $D$. 
\end{example}

In the next example we show how we can easily establish local structural identifiability of a more complicated network.

\begin{example}[Dietrich et al. \cite{DieLekTur1998}]\upshape\label{Ex:turski}
Consider a viscoelastic material studied in \cite{DieLekTur1998} and represented by a spring-dashpot network shown in Fig.~\ref{Fig:turski}(A). It can be symbolically represented by 
\begin{equation}\label{eq:ex6}
M=\Big[(((((E_1\lor\eta_1)\land E_2)\land\eta_2)\lor\eta_3)\land E_3)\lor\eta_4\Big]\land E_4.
\end{equation}
Again, we can verify the local identifiability of the above model using Table~\ref{thetable} and obtain
\[
\begin{aligned}
\cF(M)&=\Big[((((\overbrace{(A\oplus B)}^C\odot A)\odot B)\oplus B)\odot A)\oplus B\Big]\odot A\\
&=\Big[(((\overbrace{(C\odot A)}^A\odot B)\oplus B)\odot A)\oplus B\Big]\odot A\\
&=\Big[((\overbrace{(A\odot B)}^D\oplus B)\odot A)\oplus B\Big]\odot A=\ldots = D\\
\end{aligned}
\]
This simple computation confirms that the model is locally structurally identifiable.
\end{example}

Our method can also verify if a network is  unidentifiable, providing the reason for the lack of its identifiability. Consider the following example. 

\begin{example}[Unidentifiable model]\upshape
Consider a viscoelastic model used in  \cite{Ros50} and shown in Fig.~\ref{Fig:turski}(B). Using the notation previously introduced,  it can symbolically be written as 
\[
M=\Big[(((E_1\land\eta_1)\lor(E_2\land\eta_2))\land E_3\land\eta_3)\lor (E_4\land\eta_4)\Big]\land E_5\land \eta_5.
\]
Now applying Table \ref{thetable}, we obtain
\[
\begin{aligned}
\cF(M)
&=\Big[((\overbrace{(A\odot B)}^D\oplus \overbrace{(A\odot B)}^D)\odot \overbrace{(A\odot B)}^D)\oplus \overbrace{(A\odot B)}^D\Big]\odot \overbrace{A\odot B}^D\\
&=\Big[(\overbrace{(D\oplus D)}^D\odot D)\oplus D\Big]\odot D\\
&=\Big[\overbrace{(D\odot D)}^{u}\oplus D\Big]\odot D=\emph{u}.
\end{aligned}
\]
Whenever Table \ref{thetable} indicates $u$ (i.e.~the corresponding substructure is unidentifiable), this inevitably leads to the whole model being unidentifiable. Moreover, our method can also explain what is the reason for the lack of identifiability.  In this example the situation is simple: joining in series a Maxwell element (type D) with a generalized Maxwell model leads to an unidentifiable network.
\end{example}

So far we have considered only \emph{local} identifiability of mechanical systems. Now we complete the presentation and discussion of our results by introducing a criterium, which establishes when a given network is \emph{globally} structurally identifiable.

\begin{thm}[Global identifiability]\label{Thm:global}
A viscoelastic model represented by a spring-dashpot network is globally identifiable if and only if it is locally identifiable and the network is constructed by adding either in parallel or in series at the bounding nodes exactly one basic element (spring or dashpot) at a time.
\end{thm}

Note that the network given in Fig.~\ref{Fig:turski}(A) and considered in Example \ref{Ex:turski} was deemed locally structurally identifiable. We note that it can be constructed by adding just one element at a time and therefore it is \emph{globally} structurally identifiable. Similarly, all the simple models shown in Fig.~\ref{Fig:simple} can also by constructed adding only one element at a time, and since they are locally identifiable, we conclude that they are also globally structurally identifiable. Now consider a model which is locally, but not globally, structurally identifiable. 

\begin{example}[Local but not global identifiability]\upshape
Consider a generalized Kelvin-Voigt model shown  Fig.~\ref{Fig:turski}(C) and used in \cite{BugCowBea10, MahOttOlu12}) in the context of cardiovascular modeling. It can be symbolically represented by 
\[
M=E_0\land(E_1\lor\eta_1)\land(E_2\lor\eta_2)\land(E_3\lor\eta_3).
\]
Thus the local identifiability can be checked by computing
\[
\cF(M)=A\odot (A\oplus B)\odot (A\oplus B)\odot (A\oplus B)=A\odot C\odot C\odot C = A.
\]
We immediately conclude that the network is locally identifiable. In order to verify whether it is also globally identifiable, note that this network \emph{cannot} be constructed by  adding only one element at a time.  Thus the system is only locally, but not globally, identifiable.  However, in this case the non-global identifiability
arises from merely permuting the parameters among the three Voigt elements.
\end{example}

% You may title this section "Methods" or "Models". 
% "Models" is not a valid title for PLoS ONE authors. However, PLoS ONE
% authors may use "Analysis" 

\section*{Analysis}
%%%%%%%%%%%%%%%%%%%%%%%%%%%%%%%%%%%%%%

In this section, we prove the main results from the previous section.
To do this requires a careful analysis of the structure of the
constitutive equation after combining a pair of systems in series or
in parallel.

Let $N_{1}$ and $N_{2}$ be spring-dashpot models whose respective constitutive equations
are $L_1\epsilon=L_2\sigma$ and $L_3\epsilon=L_4\sigma$, where $L_i$ represent linear differential operators.  We can write the differential operators (in general form) as:

\begin{equation} \label{eqn:ldefns}
\begin{aligned}
L_1&=a_{n_1}d^{n_1}/dt^{n_1}+...+a_{m_1}d^{m_1}/dt^{m_1}\\
L_2&=b_{n_2}d^{n_2}/dt^{n_2}+...+b_{m_2}d^{m_2}/dt^{m_2}\\
L_3&=c_{n_3}d^{n_3}/dt^{n_3}+...+c_{m_3}d^{m_3}/dt^{m_3}\\
L_4&=e_{n_4}d^{n_4}/dt^{n_4}+...+e_{m_4}d^{m_4}/dt^{m_4}\\
\end{aligned}
\end{equation}

\begin{rmk} 
Table \ref{shapestable} shows that there are restrictions on the values of the $n_i$ and $m_i$, e.g.~the differential order of the lowest order term in $\sigma$ is always zero and the differential order of the lowest order term in $\epsilon$ is zero or one, but we leave the operators in general form for simplicity.
\end{rmk}

We now show the form of the resulting constitutive equation after combining these systems  in series or in parallel, in terms of these differential operators.  In what follows, we will treat the differential operators $L_i$ as polynomial functions in the variable $d/dt$.  For example, $L_1$ can be thought of as a polynomial $a_{n_1}x^{n_1}+...+a_{m_1}x^{m_1}$.   

\subsubsection*{Series connection}

Suppose that $M = N_{1} \land N_{2}$ is a series connection of models $N_{1}$ and $N_{2}$,
whose constitutive equations are 
 $L_1\epsilon_1=L_2\sigma_1$ and $L_3\epsilon_2=L_4\sigma_2$, respectively.  Then the stresses ($\sigma$) are the same for the two systems while the strains ($\epsilon$) are added.  If $L_1$ and $L_3$ are relatively prime, then the constitutive equation of $M$ is:
\begin{equation} \label{eq:series}
L_1L_3\epsilon=(L_1L_4+L_2L_3)\sigma, \quad \epsilon=\epsilon_1+\epsilon_2,\quad \sigma=\sigma_1=\sigma_2.
\end{equation}
 We assume that $a_{n_1}=c_{n_3}=1$, so that the constitutive equation is monic.  If $L_1$ and $L_3$ have a common factor, then the constitutive equation of $M$ is obtained by dividing
 (\ref{eq:series})  by the greatest common divisor of $L_1$ and $L_3$. 

\subsubsection*{Parallel connection}

Suppose that $M = N_{1} \lor N_{2}$ is a parallel connection of models $N_{1}$ and $N_{2}$,
whose constitutive equations are 
 $L_1\epsilon_1=L_2\sigma_1$ and $L_3\epsilon_2=L_4\sigma_2$, respectively.  Then the strains ($\epsilon$) are the same for the two systems while the stresses ($\sigma$) are added.  If $L_2$ and $L_4$ are relatively prime, then the constitutive equation is:
\begin{equation}\label{eq:parallel}
(L_1L_4+L_2L_3)\epsilon=L_2L_4\sigma,\quad \epsilon=\epsilon_1=\epsilon_2,\quad \sigma=\sigma_1+\sigma_2.
\end{equation}
We assume that $b_{m_2}=e_{m_4}=1$, so that the constitutive equation is monic.  If $L_2$ and $L_4$ have a common factor, then the constitutive equation is obtained by dividing
 (\ref{eq:parallel})  by the greatest common divisor of $L_2$ and $L_4$.

\subsubsection*{Types of networks}

Now we prove Proposition \ref{Prop:types}, that is, we show that every spring-dashpot network, given by equation \eqref{eq:ODE}, has one of the four possible types displayed in Table \ref{shapestable}, which are defined by the \textit{shapes} of the linear operators $L_i$ acting on $\epsilon$ and $\sigma$.  We make this notion precise:

\begin{defn} 
The \textit{shape} of a linear operator $L_i$ is a pair of numbers, written $[n_i,m_i]$, where $n_i$ is the highest differential order and $m_i$ is the lowest different order.
\end{defn} 

\begin{table}
\caption{
\bf{Possible types of constitutive equations}}
%\hspace{1cm}
\begin{center}
    \begin{tabular}{ | p{1.5cm} | p{2.3cm} | p{2.3cm} | p{3cm} | p{3cm} |}
    \hline
   Type & Shape in $\epsilon$  	& Shape in $\sigma$ 		\\ \hline
    	A  & [$n,0$] 			& [$n,0$]  					\\ \hline
    	B  & [$n+1,1$] 			& [$n,0$] 				\\ \hline
 	C  & [$n+1,0$] 			& [$n,0$] 				\\ \hline
	D  & [$n,1$] 			& [$n,0$] 					\\ \hline
    \end{tabular}
\end{center}
\begin{flushleft}
The four possible types of constitutive equations, defined by the shapes of the linear operators acting on $\epsilon$ and $\sigma$, written in brackets.
\end{flushleft}
\label{shapestable}
\end{table}

We note that a spring is of type A and a dashpot is of type B.  A Voigt element is formed by a parallel extension of types A and B, which forms type C, and a Maxwell element is formed by a series extension of types A and B, which forms type D.  The properties of these four types are displayed in Table \ref{shapestable}.  We can now form the $10$ possible combinations of pairing two of these types in series and the $10$ possible combinations of pairing two of these types in parallel.  In Tables \ref{seriestable} and \ref{paralleltable}, we show the $20$ total possibilities and demonstrate that each pairing results in a type A, B, C, or D.  Since every spring-dashpot network can be written as a combination, in series or in parallel, of springs and dashpots, then we have shown by induction that joining any two spring-dashpot networks in series or in parallel results in one of these four types.

\begin{table}
\caption{
\bf{Series connection}}
%\hspace{1cm}
\begin{center}
    \begin{tabular}{ | p{1cm} | p{2.5cm} | p{2.5cm} | p{2cm} | p{2cm} | p{2cm} | p{1cm} |}
    \hline
Type & Shape in $\epsilon$  & Shape in $\sigma$  & Non-monic coefficients & Parameters & Identifiable?  & Type  \\ \hline
    (A,A)  & [$n_1+n_2,0$] & [$n_1+n_2,0$] & $2n_1+2n_2+1$ & $2n_1+2n_2+2$ & Not Id & A  \\ \hline
		(A,B)  & [$n_1+n_2+1,1$] & [$n_1+n_2+1,0$] & $2n_1+2n_2+2$ & $2n_1+2n_2+2$ & Id & D  \\ \hline
		(A,C)  & [$n_1+n_2+1,0$] & [$n_1+n_2+1,0$] & $2n_1+2n_2+3$ & $2n_1+2n_2+3$ & Id & A  \\ \hline
		(A,D)  & [$n_1+n_2,1$] & [$n_1+n_2,0$] & $2n_1+2n_2$ & $2n_1+2n_2+1$ & Not Id & D  \\ \hline
		(B,B)  & [$n_1+n_2+1,1$] & [$n_1+n_2,0$] & $2n_1+2n_2+1$ & $2n_1+2n_2+2$ & Not Id & B  \\ \hline
		(B,C)  & [$n_1+n_2+2,1$] & [$n_1+n_2+1,0$] & $2n_1+2n_2+3$ & $2n_1+2n_2+3$ & Id & B  \\ \hline
		(B,D)  & [$n_1+n_2,1$] & [$n_1+n_2,0$] & $2n_1+2n_2$ & $2n_1+2n_2+1$ & Not Id & D  \\ \hline
		(C,C)  & [$n_1+n_2+2,0$] & [$n_1+n_2+1,0$] & $2n_1+2n_2+4$ & $2n_1+2n_2+4$ & Id & C  \\ \hline
		(C,D)  & [$n_1+n_2+1,1$] & [$n_1+n_2+1,0$] & $2n_1+2n_2+2$ & $2n_1+2n_2+2$ & Id & D  \\ \hline
		(D,D)  & [$n_1+n_2-1,1$] & [$n_1+n_2-1,0$] & $2n_1+2n_2-2$ & $2n_1+2n_2$ & Not Id & D  \\ \hline
    \end{tabular}
\end{center}
\begin{flushleft}
Two systems of types A, B, C, or D are combined in series, where in the first system $n=n_1$ and in the second system $n=n_2$.
\end{flushleft}
\label{seriestable}
\end{table}

\begin{table}
\caption{
\bf{Parallel connection}}
%\hspace{1cm}
\begin{center}
    \begin{tabular}{ | p{1cm} | p{2.5cm} | p{2.5cm} | p{2cm} | p{2cm} | p{2cm} | p{1cm} |}
    \hline
 Type & Shape in $\epsilon$  & Shape in $\sigma$  & Non-monic coefficients & Parameters & Identifiable?  & Type  \\ \hline
    (A,A)  & [$n_1+n_2,0$] & [$n_1+n_2,0$] & $2n_1+2n_2+1$ & $2n_1+2n_2+2$ & Not Id & A  \\ \hline
		(A,B)  & [$n_1+n_2+1,0$] & [$n_1+n_2,0$] & $2n_1+2n_2+2$ & $2n_1+2n_2+2$ & Id & C  \\ \hline
		(A,C)  & [$n_1+n_2+1,0$] & [$n_1+n_2,0$] & $2n_1+2n_2+2$ & $2n_1+2n_2+3$ & Not Id & C  \\ \hline
		(A,D)  & [$n_1+n_2,0$] & [$n_1+n_2,0$] & $2n_1+2n_2+1$ & $2n_1+2n_2+1$ & Id & A  \\ \hline
		(B,B)  & [$n_1+n_2+1,1$] & [$n_1+n_2,0$] & $2n_1+2n_2+1$ & $2n_1+2n_2+2$ & Not Id & B  \\ \hline
		(B,C)  & [$n_1+n_2+1,0$] & [$n_1+n_2,0$] & $2n_1+2n_2+2$ & $2n_1+2n_2+3$ & Not Id & C  \\ \hline
		(B,D)  & [$n_1+n_2+1,1$] & [$n_1+n_2,0$] & $2n_1+2n_2+1$ & $2n_1+2n_2+1$ & Id & B  \\ \hline
		(C,C)  & [$n_1+n_2+1,0$] & [$n_1+n_2,0$] & $2n_1+2n_2+2$ & $2n_1+2n_2+4$ & Not Id & C  \\ \hline
		(C,D)  & [$n_1+n_2+1,0$] & [$n_1+n_2,0$] & $2n_1+2n_2+2$ & $2n_1+2n_2+2$ & Id & C  \\ \hline
		(D,D)  & [$n_1+n_2,1$] & [$n_1+n_2,0$] & $2n_1+2n_2$ & $2n_1+2n_2$ & Id & D  \\ \hline
    \end{tabular}
\end{center}
\begin{flushleft}
Two systems of types A, B, C, or D are combined in parallel, where in the first system $n=n_1$ and in the second system $n=n_2$.
\end{flushleft}
\label{paralleltable}
\end{table}

\begin{rmk} We note that if a type B or D is combined in series with a type B or D, then $L_1$ and $L_3$ have a common factor (since both lacked a constant term), so the equation $L_1L_3\epsilon=(L_1L_4+L_2L_3)\sigma$ is divided by $\gcd(L_1,L_3) = d/dt$ to arrive at the
shapes listed in the table.  
\end{rmk}

In addition to the type of equation that results after combining two equations of types $\left\{A,B,C,D\right\}$, we have in Tables \ref{seriestable} and \ref{paralleltable} the resulting identifiability properties of each equation, which we will obtain in the next section.  Note that Definition \ref{defn:id} implies that if there are more parameters than non-monic coefficients, then the system must be unidentifiable.  The tables show that the number of non-monic coefficients is bounded by the number of parameters, thus a necessary condition for identifiability is that the number of parameters equals the number of non-monic coefficients in the constitutive equation \eqref{eq:ODE}.  In the next section, we show that this is also a sufficient condition.

%%%%%%%%%%%%%%%%%%%%%%%%%%%%%%%%%%%%%%
\subsubsection*{Local identifiability}

Consider a  spring-dashpot system $M$
whose final step connection is a series connection of two systems
$N_{1}$ and $N_{2}$, i.e.~$M = N_{1} \land N_{2}$.
Since the number of non-monic coefficients in any spring-dashpot
model is always less than or equal to the number of parameters in
that model, we know that a necessary condition for this system
to be locally identifiable is that $N_{1}$ and $N_{2}$ are both
locally identifiable. Let $L_1\epsilon_1=L_2\sigma_1$ be the constitutive equation for $N_{1}$
and $L_3\epsilon_2=L_4\sigma_2$ be the constitutive equation for $N_{2}$.
Each of the operators $L_{1}$, $L_{2}$, $L_{3}$, and $L_{4}$ will
have a fixed shape determined by the structure of $N_{1}$ and $N_{2}$. 
Assuming that $N_{1}$ and $N_{2}$ are locally identifiable,
we can choose parameters in each of the models $N_{1}$ and $N_{2}$
so that the coefficients of these constitutive equations are arbitrary
numbers.  Thus, deciding identifiability of this system amounts
to determining whether the map that takes the pair of equations 
$(L_1\epsilon_1=L_2\sigma_1, L_3\epsilon_2=L_4\sigma_2)$ to the constitutive
equation $f \epsilon = g \sigma$, where $f=L_1L_3$,  $g=L_1L_4+L_2L_3$, $\epsilon=\epsilon_1+\epsilon_2$, and $\sigma=\sigma_1=\sigma_2$ (cf. \eqref{eq:series}), for the system $M$ is finite-to-one or not.

 The same reasoning works {\it mutatis mutandis} for
parallel connections, where we now concern ourselves with the map from 
the pair of equations $(L_1\epsilon_1=L_2\sigma_1, L_3\epsilon_2=L_4\sigma_2)$
with generic coefficients to the constitutive equation 
for $M = N_{1} \lor  N_{2}$ given in \eqref{eq:parallel}.

To make the above, intuitive, statements precise we introduce the following definition.

\begin{defn}
The \emph{shape factorization problem} for a quadruple of shapes 
\[
Q=([n_{1}, m_{1}], [n_{2}, m_{2}], [n_{3}, m_{3}], [n_{4}, m_{4}])
\] 
is the following problem:  for a generic pair of 
polynomials $(f,g)$ with $f$ monic such that the 
${\rm shape}(f) = [n_{1} + n_{3}, m_{1} + m_{3}]$
and  ${\rm shape}(g) = [\max(n_{1} + n_{4}, n_{2} + n_{3}), 
\min( m_{1} + m_{4}, m_{2} + m_{3}) ]$,
do there exist finitely many quadruples of polynomials $(L_{1}, L_{2}, L_{3}, L_{4})$
with ${\rm shape}(L_{i}) = [n_{i}, m_{i}]$, $L_{1}$ and $L_{3}$ are monic,  
and such that 
$f = L_{1}L_{3}$ and $g = L_{1} L_{4} + L_{2} L_{3}$?
A quadruple of shapes $Q$ is said to be \emph{good} if
the shape factorization problem for that quadruple has a positive solution.
\end{defn}

Since the above definition introduces one of the key concepts of the paper,  in the following example we shall further  illustrate the meaning of the shape factorization problem.

\begin{ex}
Suppose that our quadruple 
\[
([n_{1}, m_{1}], [n_{2}, m_{2}], [n_{3}, m_{3}], [n_{4}, m_{4}]) =  ([2,0], [2,0], [3,0],[2,0])
\] which is a special case
of joining models of types $A$ and $C$ in series.  The shape factorization
problem in this case asks the following question:

Let $(f,g)$ be a generic pair of polynomials  where $f$ and $g$ are degree $5$ polynomials
with nonzero constant term and $f$ is monic:
\[
\begin{aligned}
&f = x^{5} + f_{4}x^{4} + f_{3}x^{3} + f_{2}x^{2} + f_{1}x + f_{0}\\
&g = g_{5}x^{5} + g_{4}x^{4} + g_{3}x^{3} + g_{2}x^{2} + g_{1}x + g_{0}. 
\end{aligned}
\]
Do there exist finitely many polynomials
\[
\begin{aligned}
&L_{1} = x^{2} + a_{1}x + a_{0},  			&\quad& L_{2} = b_{2}x^{2} + b_{1}x + b_{0}\\
&L_{3} = x^{3}+ c_{2} x^{2} + c_{1}x + c_{0},  	&& L_{4} = d_{2}x^{2} + d_{1}x + d_{0}
\end{aligned}
\]
such that $f = L_{1}L_{3}$ and $g = L_{1}L_{4} + L_{2}L_{3}$?
Or to say it another way, for generic values of $f_{4}, \ldots, f_{0}$ and $g_{5}, \ldots, g_{0}$,
does the system of $11$ equations in $11$ unknowns:
\begin{eqnarray*}
f_{4} & = & a_{1} + c_{2}  \\
f_{3} & = &  a_{0} + a_{1}c_{2} + c_{1}  \\
f_{2} & = &  a_{0}c_{2} + a_{1}c_{1} + c_{0}  \\
f_{1} & = &  a_{0}c_{1} + a_{1} c_{0} \\
f_{0} & = & a_{0}c_{0}  \\
g_{5}  & = & b_{2}  \\
g_{4}  & = & b_{1} + b_{2} c_{2} + d_{2} \\
g_{3}  & = & b_{0} + b_{1}c_{2} + b_{2}c_{1} + a_{1}d_{2} + d_{1} \\
g_{2}  & = &  b_{0}c_{2} + b_{1}c_{1} + b_{2}c_{0} + a_{0}d_{2} + a_{1}d_{1} + d_{0}  \\
g_{1}  & = &  b_{0}c_{1} + b_{1}c_{0} + a_{0}d_{1} +a_{1}d_{0}  \\
g_{0}  & = &  b_{0}c_{0} + a_{0}d_{0}
\end{eqnarray*}
have only finitely many solutions?
\end{ex}

The language of shape factorization problems and the remarks in the preceding
paragraphs allow us to reduce the local identifiability  problem for a spring-dashpot system to determining whether a
certain quadruple is a good quadruple.

\begin{prop}
Let $M = N_{1} \land N_{2}$ be a spring-dashpot model joined in series
from $N_{1}$ and $N_{2}$, where $N_{1}$ has constitutive equation
$ L_{1} \epsilon_1 = L_{2} \sigma_1$ of shapes $[n_{1}, m_{1}]$ and $[n_{2}, m_{2}]$,
respectively,  and   $N_{2}$ has constitutive equation
$ L_{3} \epsilon_2 = L_{4} \sigma_2$ of shapes $[n_{3}, m_{3}]$ and $[n_{4}, m_{4}]$,
respectively.  Then the model $M$ is locally identifiable if and
only if
\begin{enumerate}
\item  $N_{1}$ and $N_{2}$ are locally identifiable, and
\item  $([n_{1}, m_{1}], [n_{2}, m_{2}], 
[n_{3}, m_{3}], [n_{4}, m_{4}])$ is a good quadruple.
\end{enumerate}
Similarly, if $M = N_{1} \lor N_{2}$ is a spring-dashpot model joined in parallel
from $N_{1}$ and $N_{2}$, then $M$ is locally identifiable if and only if
\begin{enumerate}
\item  $N_{1}$ and $N_{2}$ are locally identifiable, and
\item  $([n_{2}, m_{2}], [n_{1}, m_{1}],  
 [n_{4}, m_{4}], [n_{3}, m_{3}])$ is a good quadruple.
\end{enumerate}
\end{prop}

So what remains to show is that, for the shapes that arise 
in spring-dashpot models, whether a quadruple of shapes is
a good quadruple only depends on the types ($A, B, C, $ or $D$) of
the systems being combined.  The proof of this statement
will occupy the rest of this section.

Let $f$ and $g$ be two polynomials.  Note that for given fixed shapes, 
$[n_{1}, m_{1}]$ and $[n_{3}, m_{3}]$, there are at most finitely
many factorizations $f = L_{1}L_{3}$, where $L_{1}$ has shape 
$[n_{1}, m_{1}]$ and $L_{3}$ has shape $[n_{3}, m_{3}]$ and both are monic.  
This is
because there are at most finitely many ways to factorize a monic polynomial into
monic factors.
Once we fix one of these finitely many choices for $L_{1}$ and
$L_{3}$, the equation $ g = L_{1} L_{4} + L_{2} L_{3}$ is
a linear system in the (unknown) coefficients of $L_{2}$ and $L_{4}$.

For a polynomial $f = f_{n}x^{n} + \cdots + f_{m}x^{m}$ of shape $[n,m]$, we can write
the coefficients of $f$ as a vector, which we denote 
\[
[f] := \begin{pmatrix}
f_{n} \\
\vdots \\
f_{m} \end{pmatrix}.
\]

Let $L_i$ have shape $[n_i,m_i]$, as defined in Equation \eqref{eqn:ldefns}. The vector of coefficients of $L_1L_4$ can be written 
as the result of a matrix vector product as:
\[[L_1L_4]=
\begin{pmatrix}
a_{n_1} & 0 & \cdots & 0 \\
\vdots & a_{n_1} & \cdots & 0 \\
a_{m_1} & \vdots & \cdots & \vdots \\ 
0 & a_{m_1} & \cdots & 0 \\
\vdots & 0 & \cdots & a_{n_1} \\
\vdots & \vdots & \cdots & \vdots \\
0 & 0 & \cdots & a_{m_1}  
\end{pmatrix}
\begin{pmatrix}
e_{n_4} \\
\vdots \\
e_{m_4}
\end{pmatrix}.
\]
We will refer to this product as $H'[L_4]$, where $H'$ is a $n_1+n_4-m_1-m_4+1$ by $n_4-m_4+1$ matrix.  
Likewise, the coefficients of $L_2L_3$ can be written as the result of a matrix vector product as:
\[
[L_2L_3]=
\begin{pmatrix}
c_{n_3} & 0 & \cdots & 0 \\
\vdots & c_{n_3} & \cdots & 0 \\
c_{m_3} & \vdots & \cdots & \vdots \\ 
0 & c_{m_3} & \cdots & 0 \\
\vdots & 0 & \cdots & c_{n_3} \\
\vdots & \vdots & \cdots & \vdots \\
0 & 0 & \cdots & c_{m_3}  
\end{pmatrix}
\begin{pmatrix}
b_{n_2} \\
\vdots \\
b_{m_2}
\end{pmatrix}.
\]
We will refer to this product as $G'[L_2]$, where $G'$ is a $n_2+n_3-m_2-m_3+1$ by $n_2-m_2+1$ matrix. Then we call the \textit{matrix factored form} of $[L_1L_4+L_2L_3]$ the expression:
\begin{equation} \label{eqn:factored}
G[L_2]+H[L_4],
\end{equation}
where the matrices $G$ and $H$ are the matrices $G'$ and $H'$ padded with rows of zeros so
that coefficients corresponding to monomials of the
same degree appear in the same row.   This makes  $(G \ \ H)$  a $\max\left\{n_1+n_4,n_2+n_3\right\}-\min\left\{m_1+m_4,m_2+m_3\right\}+1$ by $n_2-m_2+n_4-m_4+2$ matrix. 

We can now state a criteria for determining if the shape factorization problem has  finitely many solutions:

\begin{prop} \label{prop:abinvert}
The quadruple $([n_{1}, m_{1}], [n_{2}, m_{2}], 
[n_{3}, m_{3}], [n_{4}, m_{4}])$ is a good quadruple if and only
if the matrix $(G \ \ H)$ is generically invertible.  
\end{prop}

\begin{proof}
We can write the shape factorization problem of type $([n_{1}, m_{1}], [n_{2}, m_{2}], 
[n_{3}, m_{3}], [n_{4}, m_{4}])$ in matrix factored form as $G[L_2]+H[L_4]=[g]$ (see \eqref{eqn:factored}), so that 
\[
\begin{pmatrix}
G & H
\end{pmatrix}
\begin{pmatrix}
L_2 \\
L_4
\end{pmatrix}
= [g].
\]
This system  has a unique solution if and only if $(G \ \ H)$ is generically invertible, 
i.e.~invertible for a generic choice of parameter values.
\end{proof}

\smallskip

\begin{ex}\label{ex:quadruple}
Suppose that our quadruple $([n_{1}, m_{1}], [n_{2}, m_{2}], 
[n_{3}, m_{3}], [n_{4}, m_{4}])$ is
$([2,0], [2,0], [3,0],[2,0])$, which is a special case of joining models of types $A$ and $C$ in series.  The resulting
matrix $(G \ \ H)$ is the matrix
$$
\left(
\begin{array}{cccccc}
0 & 0     & 0     & c_{3} & 0     & 0     \\
a_{2} & 0 & 0     & c_{2} & c_{3} & 0     \\
a_{1} & a_{2} & 0 & c_{1} & c_{2} & c_{3} \\
a_{0}     & a_{1} & a_{2} & c_{0} & c_{1} & c_{2} \\
0     & a_{0}     & a_{1} & 0     & c_{0} & c_{1} \\
0     & 0     & a_{0}     & 0     & 0     & c_{0}
\end{array} \right).
$$

\end{ex}

\smallskip

 We now determine when this matrix $(G \ \ H)$ is generically invertible, i.e.~square and full rank.  The \textit{Sylvester matrix} associated to two polynomials $p(z)=p_{0}+p_{1}z+p_{2}z^2+...+p_{m}z^m$ and $q(z)=q_{0}+q_{1}z+q_{2}z^2+...+q_{n}z^n$ is the $n+m$ by $n+m$ matrix that has the coefficients of $p(z)$ repeated $n$ times as columns and the coefficients of $q(z)$ repeated $m$ times as columns in the following way:

\[
\begin{pmatrix}
p_{m} & 0 & \cdots & 0 & q_{n} & 0 & \cdots & 0\\
\vdots & p_{m} & \cdots & 0 & \vdots & q_{n} & \cdots & 0\\
p_{0} & \vdots & \cdots & \vdots & q_{0} & \vdots & \cdots & \vdots\\ 
0 & p_{0} & \cdots & 0 & 0 & q_{0} & \cdots & 0 \\
\vdots & 0 & \cdots & p_{m} & \vdots & 0 & \cdots & q_{n} \\
\vdots & \vdots & \cdots & \vdots & \vdots & \vdots & \cdots & \vdots\\
0 & 0 & \cdots & p_{0} & 0 & 0 & \cdots & q_{0}
\end{pmatrix}.
\]

$$ \underbrace{\ \ \ \ \ \ \ \ \ \ \ \ \ \ \ \ \ \ \ \ }_{n} \underbrace{\ \ \ \ \ \ \ \ \ \ \ \ \ \ \ \ \ \ \ }_{m} $$

The determinant of the Sylvester matrix of the two polynomials $p$ and $q$ is the \emph{resultant}, which is zero if and only if the two polynomials have a common root.  In particular, for generic polynomials $p$ and $q$, the Sylvester matrix is invertible \cite[Chapter 3]{Cox2005}.

We will use the Sylvester matrix in the following way.  We will show that there are submatrices of $(G \ \ H)$ that correspond to the Sylvester matrix associated to $L_1$ and $L_3$. 
%\begin{defn} 
%The \textit{Sylvester degree} of the polynomial $L_i$ is $n_i-m_i$, i.e. the degree of the shifted polynomial if the lowest degree term was the constant term and the highest degree was $n_i-m_i$.
%\end{defn}

\begin{prop} \label{prop:squarefullrank} If the matrix $(G \ \ H)$ is square, then it is generically invertible.  
\end{prop}

\begin{proof}
We claim that the columns of $(G \ \ H)$ can be ordered so that the
resulting matrix has the shape
\begin{equation}\label{eq:sylvesterish}
\begin{pmatrix}
S' &  0 & 0 \\
X & S & Y \\
0 & 0 & S''  
\end{pmatrix}
\end{equation}
where $S$ is the Sylvester matrix associated to the nonzero coefficients of 
$L_{1}$ and $L_{3}$.  Note that this means that we might shift the coefficients
down if necessary so there are no extraneous zero terms of low degree (i.e.~if the shape
is $[n_{i},m_{i}]$ with $m_{i} \neq 0$).
The matrix  $S'$ is a square lower triangular
matrix with nonzero entries on the diagonal, and $S''$ is a square
upper triangular matrix with nonzero entries on the diagonal.  
This will
prove that $(G \ \ H)$ is invertible, since its determinant will be the product
of the determinants of $S$, $S'$ and $S''$, all of which are nonzero.
To prove that claim requires a careful case analysis.

The number of columns of $(G \ \ H)$ is
$n_{4} - m_{4} +n_{2} - m_{2} + 2$ and the number of rows is
$\max( n_{1} + n_{4}, n_{2} + n_{3}) - \min(m_{1} + m_{4}, m_{2} + m_{3}) + 1$.
Without loss of generality, we can assume that the maximum is attained by $n_{1} + n_{4}$.
We need to distinguish between the two cases where the minimum is attained by $m_{1} + m_{4}$ and by $m_{2} + m_{3}$.

\medskip

\noindent{\bf Case 1: $ \min\big[m_{1} + m_{4}, m_{2} + m_{3}\big] = m_{1} + m_{4}$.}   
Since $(G \ \ H)$ is a square matrix, this implies that $n_{1} - m_{1} = 
n_{2} - m_{2} + 1$.  In this case we group the columns of  $(G \ \ H)$
in the following order.
\begin{enumerate}
\item The first $n_{1} + n_{4} - n_{2} - n_{3}$ columns of $G$
\item Then the next $n_{3} - m_{3}$ columns of $G$
\item Then all $n_{2} - m_{2} + 1 (= n_{1} - m_{1})$ columns of $H$
\item  Then the remaining $ m_{2} + m_{3} - m_{1} - m_{4}$ columns of $G$.
\end{enumerate}
This choice has the property that the middle two blocks of columns together
have the desired form, since we have chosen to start including columns from
$G$ and $H$ precisely when they both have nonzero entries in the same rows,
and stopping the formation of these when they stop having nonzero entries in the
same rows, which has the correct form.  Note we have used all columns of
$G$ since 
\[
n_{1} + n_{4} - n_{2} - n_{3} +n_{3} - m_{3} +
m_{2} + m_{3} - m_{1} - m_{4}  = 
(n_{1} - m_{1}) + (n_{4} - m_{4}) - (n_{2} - m_{2}) =
n_{4} - m_{4} + 1.
\]

\medskip

\noindent{\bf Case 2: $ \min\big[m_{1} + m_{4}, m_{2} + m_{3}\big] = m_{2} + m_{3}$.}  
Note that since $(G \ \ H)$ is square, this implies that
$n_{1} - m_{3} = n_{2} - m_{4} + 1$.
In this case, we do not need to reorder the columns to obtain the desired
form. 

We mention how to block the columns to obtain the desired form.
\begin{enumerate}
\item The first $n_{1} + n_{4} - n_{2} - n_{3}$ columns of $G$
\item Then the next $n_{3} - m_{3}$ columns of $G$
\item Then the first $n_{1} - m_{1}$ columns of $H$
\item  Then the remaining $ m_{1} + m_{4} - m_{2} - m_{3}$ columns
of $H$.  
\end{enumerate}
Note that we have the desired number of columns from the second and third
blocks, and we have chosen them so that that those columns have nonzero
entries at exactly the same rows.  Furthermore, we have used
all columns of $G$ since
$$n_{1} + n_{4} - n_{2} - n_{3} + n_{3} - m_{3} = 
n_{1} + n_{4} - n_{2} - m_{3}  =  n_{4} - m_{4} + 1$$
and all columns of $H$ since
$$n_{1} - m_{1} + m_{1} + m_{4} - m_{2} - m_{3}  = n_{1} + m_{4} - m_{2} - m_{3} = 
n_{2} - m_{2} + 1.  \qedhere$$
\end{proof}

\begin{ex}  
We can rewrite the matrix in Example \ref{ex:quadruple} as 
\[
\left(
\begin{array}{c|ccccc}
c_{3} & 0     & 0     & 0     & 0     & 0     \\
\hline
c_{2} & a_{2} & 0     & 0     & c_{3} & 0     \\
c_{1} & a_{1} & a_{2} & 0     & c_{2} & c_{3} \\
c_{0} & a_{0} & a_{1} & a_{2} & c_{1} & c_{2} \\
0     & 0     & a_{0} & a_{1} & c_{0} & c_{1} \\
0     & 0     & 0     & a_{0} & 0     & c_{0}
\end{array} \right).
\]
Here the the $5 \times 5$ matrix in the lower righthand corner
is the Sylvester matrix, the matrix $S'$ is the $1 \times 1$ matrix in the upper lefthand corner, and the matrix $S''$ is the
empty matrix.
\end{ex}

\begin{proof} [Proof of Theorem \ref{thm:parameqcoeff}] 
We will show that if the number of parameters equals the number of non-monic coefficients, then the matrix $(G \ \ H)$ is square. By Propositions \ref{prop:abinvert} and  \ref{prop:squarefullrank}, this will imply that the model is locally identifiable.
  
Let $M = N_{1} \land N_{2}$ be a spring-dashpot model joined in series
from $N_{1}$ and $N_{2}$, where $N_{1}$ has constitutive equation
$ L_{1} \epsilon_1 = L_{2} \sigma_1$ of shapes $[n_{1}, m_{1}]$ and $[n_{2}, m_{2}]$,
respectively,  and   $N_{2}$ has constitutive equation
$ L_{3} \epsilon_2 = L_{4} \sigma_2$ of shapes $[n_{3}, m_{3}]$ and $[n_{4}, m_{4}]$,
respectively.
By induction, we can 
assume that the number of parameters equals the number of non-monic coefficients for the systems $N_{1}$ and $N_{2}$, i.e.~there are $n_1-m_1+n_2-m_2+1$ parameters in the first and $n_3-m_3+n_4-m_4+1$ in the second.   Assume the number of parameters equals the number of non-monic coefficients in this full system, i.e.~
\begin{eqnarray*}
\lefteqn{
n_1-m_1+ n_2-m_2+n_3-m_3+n_4-m_4 + 2  =}   \\
&  &  
\max\left\{n_1+n_4,n_2+n_3\right\}-\min\left\{m_1+m_4,m_2+m_3\right\}+1+n_1-m_1+n_3-m_3.
\end{eqnarray*}
 Subtracting $n_1-m_1+n_3-m_3$ from both sides, we get that 
$$
n_2-m_2+n_4-m_4+2 = \max\left\{n_1+n_4,n_2+n_3\right\}-\min\left\{m_1+m_4,m_2+m_3\right\}+1.
$$  
From the definition of $(G \ \ H)$, this means the number of rows equals the number of columns, so that $(G \ \ H)$ is square.  

The argument for the parallel extension is identical and is omitted. 
\end{proof}

\begin{proof} [Proof of Theorem \ref{thm:localMT}] 
Theorem \ref{thm:parameqcoeff} shows that the model is locally identifiable if and only if the number of parameters equals the number of non-monic coefficients.  Thus the identifiability properties of the $20$ cases in Tables \ref{seriestable} and \ref{paralleltable} are determined by checking if the numbers in the columns corresponding to the number of parameters and the number of non-monic coefficients are equal.
\end{proof}

%%%%%%%%%%%%%%%%%%%%%%%%%%%%%%%%%%%%%%

\subsubsection*{Global identifiability}

We now determine necessary and sufficient conditions for global identifiability.

\begin{prop}
Let $M = N_{1} \land N_{2}$ be a spring-dashpot model joined in series
from $N_{1}$ and $N_{2}$, where $N_{1}$ has constitutive equation
$ L_{1} \epsilon_1 = L_{2} \sigma_1$ of shapes $[n_{1}, m_{1}]$ and $[n_{2}, m_{2}]$,
respectively,  and   $N_{2}$ has constitutive equation
$ L_{3} \epsilon_2 = L_{4} \sigma_2$ of shapes $[n_{3}, m_{3}]$ and $[n_{4}, m_{4}]$,
respectively.  Then the model $M$ is globally identifiable if and
only if
\begin{enumerate}
\item  $N_{1}$ and $N_{2}$ are globally identifiable,
\item  The shape factorization problem for the quadruple $([n_{1}, m_{1}], [n_{2}, m_{2}], 
[n_{3}, m_{3}], [n_{4}, m_{4}])$ generically has a unique solution.
\end{enumerate}
Similarly, if $M = N_{1} \lor N_{2}$ is a spring-dashpot model joined in parallel
from $N_{1}$ and $N_{2}$, then $M$ is globally identifiable if and only if
\begin{enumerate}
\item  $N_{1}$ and $N_{2}$ are globally identifiable, and
\item   The shape factorization problem for the quadruple $([n_{2}, m_{2}], [n_{1}, m_{1}],  
 [n_{4}, m_{4}], [n_{3}, m_{3}])$ generically has a unique solution.
\end{enumerate}
\end{prop}

\begin{proof}
We handle the case of series extensions, parallel extensions being identical.
Let $M = N_{1} \land N_{2}$.  Clearly, $N_{1}$ and $N_{2}$ must be
globally identifiable otherwise we could give two sets of parameters
yielding the same constitutive equation for $N_{1}$, which could then
be combined with parameters for $N_{2}$ to get two sets of parameters
for $M$ yielding the same constitutive equation.

Now if the shape factorization problem has a unique solution, there is
a unique way to take the constitutive equation for $M$ and solve for the
constitutive equations for $N_{1}$ and $N_{2}$, since $N_{1}$ and $N_{2}$
are globally identifiable, there is a unique way to solve for parameters
of those models giving a unique solution for parameters for $M$.  Conversely,
if there were multiple solutions to the shape factorization problem, then
by global identifiability of $N_{1}$ and $N_{2}$, we could solve all the
way back to get multiple parameter choices for the same parameter choice for $M$.
\end{proof}

Note that in our analysis of the shape factorization problem in the
previous section, we saw that once $L_{1}$ and $L_{3}$ are chosen among all their finitely
many values, when the model is locally identifiable there is a unique way to
then construct $L_{2}$ and $L_{4}$.  Hence, the shape factorization
problem has a unique solution when there is a unique way to 
factor $f = L_{1}L_{3}$.  This happens if and only if either $n_{1} = m_{1}$
or $n_{3} = m_{3}$, otherwise, generically, we can exchange roots of $L_{1}$
and $L_{3}$ giving multiple solutions.

\begin{cor} \label{cor:globidconstzero} 
Suppose that $M = N_{1} \land N_{2}$ is globally identifiable.
Then either $N_{1}$ or $N_{2}$ must have been one of a spring, a dashpot, or a Maxwell model.
Suppose that $M = N_{1} \lor N_{2}$ is globally identifiable.
Then either $N_{1}$ or $N_{2}$ must have been one of a spring, a dashpot, or a Voigt model.
\end{cor}

\begin{proof} 
The four models given by the spring, dashpot, Voigt, and Maxwell elements are
the only four locally identifiable models that have the property that
at least one of the differential operators in its constitutive equation
has exactly one term.  This can be seen by analyzing the four types
(A,B,C,D) and looking at all possibilities that arise on combining two
equations.  Once both operators do not have a single term, no model combined from such a model
can have an operator with a single term.

The three choices for the series connection (spring, a dashpot, or a Maxwell model)
are the three of four models that put a differential operator with a single
term in the correct place so there could be a unique solution to the shape factorization
probelm.  Similarly for the parallel connection. 
\end{proof}

\begin{proof} [Proof of Theorem \ref{Thm:global}]
Clearly a globally identifiable model is locally identifiable. 
By Corollary \ref{cor:globidconstzero}, we must be able to construct such
a globally identifiable model by adding at each step either a spring, dashpot,
Maxwell or Voigt element at each step, but when adding a Maxwell 
element it must be used in series
and when using a Voigt element it must have been added in parallel.  However,
adding a Maxwell element in series can be achieved by adding a spring and then a dashpot both
in series.  Similarly, adding a Voigt element in parallel can be achieved by adding
a spring and then a dashpot both in parallel.  Hence, we can work only adding
springs or dashpots at each step. 
\end{proof}

% Do NOT remove this, even if you are not including acknowledgments
\section*{Acknowledgments}
Adam Mahdi was partially supported by the VPR project under NIH-NIGMS grant \#1P50GM094503-01A0 sub-award to NCSU.
Nicolette Meshkat was partially supported by the David and
Lucille Packard Foundation. Seth Sullivant was partially supported by the David and Lucille Packard Foundation and the US National Science Foundation (DMS 0954865).

% Bibliography
%\bibliography{bibliog}

\comment{
\section*{Figure Legends}
%\begin{figure}[!ht]
%\begin{center}
%%\includegraphics[width=4in]{figure_name.2.eps}
%\end{center}
%\caption{
%{\bf Bold the first sentence.}  Rest of figure 2  caption.  Caption 
%should be left justified, as specified by the options to the caption 
%package.
%}
%\label{Figure_label}
%\end{figure}

\section*{Tables}
%\begin{table}[!ht]
%\caption{
%\bf{Table title}}
%\begin{tabular}{|c|c|c|}
%table information
%\end{tabular}
%\begin{flushleft}Table caption
%\end{flushleft}
%\label{tab:label}
% \end{table}
}

\end{document}